\pdfoutput=1
\documentclass[12pt, english]{amsart}
\usepackage{latexsym}
\usepackage{amscd, amsfonts, eucal, mathrsfs, amsmath, amssymb, amsthm}
\usepackage[utf8]{inputenc}
\input xy
\xyoption{all}

\usepackage{newcent}
\usepackage{comment}
\PassOptionsToPackage{hyphens}{url}\usepackage{hyperref}

\usepackage[margin=1in, footskip=0.3in]{geometry}

\usepackage{tikz-cd}

\newcommand{\field}[1]{\mathbf #1}
\newcommand{\mf}[1]{\mathfrak #1}
\newcommand{\mc}[1]{\mathcal #1}
\newcommand{\ms}[1]{\mathscr #1}
\newcommand{\widebar}[1]{\overline{#1}}

\DeclareMathOperator{\Jac}{Jac}
\DeclareMathOperator{\Alb}{Alb}
\DeclareMathOperator{\Sym}{Sym}
\DeclareMathOperator{\V}{\mathbf{V}}

\newcommand{\R}{\field R}

\newcommand{\C}{\field C}
\newcommand{\F}{\field F}
\newcommand{\Z}{\field Z}
\newcommand{\Q}{\field Q}

\DeclareMathOperator{\Spec}{Spec}

\newcommand{\A}{\field A}

\DeclareMathOperator{\Pic}{Pic}

\DeclareMathOperator{\GL}{GL}

\DeclareMathOperator{\Ext}{Ext}

\DeclareMathOperator{\NS}{NS}

\newcommand{\G}{\field G} 

\renewcommand{\H}{\operatorname{H}}

\newcommand{\Gal}{\operatorname{Gal}}

\DeclareMathOperator{\ind}{ind}

\DeclareMathOperator*{\tensor}{\otimes}

\DeclareMathOperator{\Br}{\operatorname{Br}}

\newcommand{\geom}{\text{\rm geom}}

\renewcommand{\mathbb}{\mathbf}

\newcommand{\Qlbar}{\overline\Q_\ell}
\newcommand{\Fq}{\F_q}

\numberwithin{equation}{subsection}

\newtheorem{thm}{Theorem}[subsection]

\newtheorem{prop}[thm]{Proposition}

\newtheorem{cor}[thm]{Corollary}

\newtheorem{lem}[thm]{Lemma}

\theoremstyle{definition}
\newtheorem{defn}[thm]{Definition}

\newtheorem{assumption}[thm]{Assumption}
\newtheorem{observation}[thm]{Observation}

\theoremstyle{remark}

\newtheorem{ques}[thm]{Question}
\newtheorem{question}[thm]{Question}

\numberwithin{equation}{subsection}
\DeclareMathOperator{\Sp}{Sp}

\author{Wei Ho and Max Lieblich}

\title{Splitting Brauer classes using the universal Albanese}

\begin{document}
\maketitle
\begin{abstract}
  We prove that every Brauer class over a field splits over a torsor under an
  abelian variety. If the index of the class is not congruent to 2
  modulo 4, we show 
  that the Albanese variety of any smooth curve of positive genus that
  splits the class also 
  splits the class, and there exist many such curves splitting the class.
  We show that this can be false when the index is congruent to 
  $2$ modulo $4$, but adding a single genus $1$ factor to the Albanese
  suffices to split the class.
\end{abstract}

\setcounter{tocdepth}{1}

\section{Introduction}
\label{sec:introduction}

Our main goal in this note is to construct torsors under abelian
varieties that split Brauer classes over fields. For a variety $X$
over a field $K$ and a Brauer class $\alpha \in \Br(K)$, let
$\alpha_X$ denote the pullback of $\alpha$ to $\Br(X)$. We say that
$X$ {\em splits} $\alpha$ if $\alpha_X$ is trivial, or equivalently,
if there is a rational map from $X$ to any Brauer-Severi variety
with cohomology class $\alpha$.

\begin{thm}\label{thm:existence}
  Given a field $K$ and a Brauer class $\alpha\in\Br(K)$, there exists
  a torsor $T$ under an abelian variety over $K$ such that
  $\alpha_T=0$. Equivalently, the torsor $T$ admits a rational map to
  any Brauer-Severi variety $V$ associated to $\alpha$.
\end{thm}

In fact, we show that there are many such torsors splitting a given Brauer
class by studying Albanese varieties. For a smooth proper
geometrically connected curve $C$ over $K$, let
$C \to \Alb_C := \Pic^1_{C/K}$ denote the Albanese morphism for $C$,
so the Albanese variety $\Alb_C$ is a torsor under the Jacobian variety
$\Jac_C := \Pic^0_{C/K}$.

\begin{thm} \label{thm:main} Let $K$ be a field and
  $\alpha \in \Br(K)$. Write $\ind(\alpha)$ for the index of $\alpha$.
\begin{enumerate}
\item If $\ind(\alpha)\not\equiv 2\pmod 4$, then for any smooth proper
  geometrically connected 
  curve $C$ over $K$ of positive genus such that $\alpha_C=0$, we have
  that $\alpha_{\Alb_C}=0$.
\item If $\ind(\alpha)\equiv 2\pmod 4$, then for any smooth proper
  geometrically connected 
  curve $C$ over $K$ of positive genus such that $\alpha_C=0$, there
  exists a genus $1$ curve $C'$ over $K$ such that
  $\alpha_{C'\times\Alb_C}=0$.
  \end{enumerate}
\end{thm}

Note that there are many curves splitting any given Brauer class $\alpha$
of index $m$. For example, for $m \geq 3$, any complete intersections of $m-2$ sections
of the anticanonical sheaf of an associated $(m-1)$-dimensional
Brauer-Severi variety will split the class, and a general such
complete intersection is a smooth proper geometrically connected curve
by Bertini's theorem. (This argument needs to be slightly tweaked if $K$ is
finite, but in that case the theorems above are trivial since $\Br(K)=0$.)
Thus, Theorem \ref{thm:existence} follows immediately from Theorem \ref{thm:main}.

This result grew out of considering the following well-known question.

\begin{ques}\label{ques:main}
  Given a field $K$ and a Brauer class $\alpha\in\Br(K)$, is there a
  genus 1 curve $C$ over $K$ such that $\alpha_C=0$?
\end{ques}

This question was asked explicitly by Pete L.~Clark on his website and
in \cite{openproblems}, and an affirmative answer was given
when $\alpha$ has index $3$ by Swets
\cite{Swets}, index at most $5$ in \cite{ajdjwh}, and index $6$
under some assumptions on $K$ by Auel \cite{auel-video}. The techniques in
\cite{ajdjwh,auel-video} are unlikely to generalize to higher index, however, and
the general question seems quite difficult.

To prove Theorem \ref{thm:main}, we use the basic theory of big
monodromy to deform a curve splitting the class $\alpha$ to a curve
whose Jacobian has minimal N\'eron-Severi group. When the N\'eron-Severi group
is minimal, it is relatively easy to compare
obstruction classes for sections of the Picard scheme of the curve
with sections of the Picard scheme of its Albanese, giving the result
for the general curve. Specializing back to the Albanese of the original curve
finishes the proof.

\begin{assumption}
  Since the Brauer group of a finite field is trivial, we assume from
  now on that $K$ is an infinite field.
\end{assumption}

\subsection*{Acknowledgments}
We thank Benjamin Antieau, Asher Auel, Bhargav Bhatt, Daniel Bragg, Pete L.~Clark,
Brendan Creutz, Jean-Louis Colliot-Th\'el\`ene, Aise Johan de Jong, Gavril
Farkas, Charles Godfrey, Marc Hindry, Daniel Krashen, Davesh Maulik, Hirune Mendebaldeko,
Ben Moonen, Brian Osserman, David Saltman, Alexei Skorobogatov,
Bianca Viray, and Kristin de Vleming for helpful
comments during the preparation of this note.

We thank Hindry and Skorobogatov for pointing us to relevant work
of Zarhin. We thank Moonen for bringing the appendix to \cite{1711.07722}
to our attention. We thank Antieau and Auel for
finding an error in our treatment of classes of period $2$ and suggesting a
fix. After receiving our preprint, Antieau and Auel also devised an alternative
proof of the main splitting result using work on the stable birational geometry
of symmetric powers of Brauer--Severi varieties; we thank them for letting us
reproduce a sketch of their argument here. WH and ML were supported by NSF
grants DMS-1701437 and DMS-1600813, respectively.

\section{The Jacobian of the general curve}
\label{sec:idea}

Our overarching goal in this section is to describe a proof of the
following folk theorem. Write $\ms M_g$ for the stack of smooth proper
geometrically connected curves of genus $g\geq 2$. This is a smooth algebraic
stack over $\Spec\Z$ with irreducible geometric fibers (see, e.g.,
\cite{delignemumford}). 

\begin{prop}\label{prop:univ-jac}
  Suppose $C$ is a smooth proper geometrically connected
  curve of genus $g$ over an algebraically closed field $k$ such that the
  induced map $\Spec k\to\ms M_g$ sends the point of $\Spec k$
  to the generic point of a fiber of $\ms M_g\to\Spec\Z$. Then we have
  that $\NS(\Jac_C)=\Z\Theta$, where $\Theta$ is the class of the
  $\Theta$-divisor on $\Jac_C$. 
\end{prop}

The proof can be achieved using the theory of big monodromy in the
Hodge or $\ell$-adic context. The $\ell$-adic proof applies in all
characteristics, while the Hodge-theoretic proof only applies in
characteristic $0$. Since we assume that the reader may not
be intimately familiar with the theory, we briefly sketch both
arguments here. The results that prove this Proposition are Corollary
\ref{cor:big-ell-mono} and Corollary \ref{cor:big-mono} below. We
also refer the reader to another exposition via $\ell$-adic methods
by Moonen in the appendix
\cite{moonen-appendix} to \cite{1711.07722}.
There is also a direct proof using results of Zarhin \cite{MR1748293, MR2131907}
(building on earlier work of Mori); see Section \ref{sec:zarhin}.

\subsection{Representation theory and notation}
\label{sec:notation}

We will use the following notation and results in this section.
\begin{enumerate}
\item Given a ring $R$, we write
  $\Sp(2g, R)$ for the symplectic group associated to the standard
  symplectic form of dimension $2g$. We write $\GL(n, R)$ for the
  algebraic general linear group over $R$ (not just the $R$-points). Given a
  field $L$ and an $L$-vector space $V$, we will write $\GL(V)$ for
  the algebraic group of automorphisms of $V$ (not just the
  $L$-points). This is non-canonically isomorphic to $\GL(\dim_L V,L)$.
\item\label{monod-gp} Given an abstract group $\pi$ and a representation
  $\rho\colon\pi\to\GL(n, L)$ over a field $L$, we will write
  $G(\rho)\subset\GL(n, L)$ for the connected component of the
  identity of the Zariski closure of the image of $\rho$. If $\mathbf
  V\to B$ is a local system of $L$-vector spaces on a topological space, we will write
  $G(\mathbf V)$ for $G(\rho)$, where
  $\rho\colon\pi_1(B,b)\to\GL(V)$ is the monodromy representation attached to
  $\mathbf V$ and $V = \mathbf V_b$.
\item\label{repn-thry} Let $V$ be a vector space over a field $L$ of dimension $2g$ equipped with
  the standard symplectic form. The  pairing defines an $\Sp(2g, L)$-invariant map
  $\bigwedge^2V\to L$. The kernel $V_2\subset\bigwedge^2 V$ of the pairing map is an absolutely
  irreducible representation of $\Sp(2g, L)$; see \cite[Theorem 17.5]{fultonharris}.

\end{enumerate}

\subsection{The N\'eron-Severi sheaf}
\label{sec:neron-severi-sheaf}

Suppose $X\to\Spec k$ is a smooth proper geometrically connected
variety over a field $k$. (The theory we describe here generalizes, but we
avoid such a digression.) Let $\Pic_{X/k}$ be the
Picard scheme of $X$ over $k$ and $\Pic^0_{X/k}$ its connected
component. When $X$ is a curve, the Jacobian variety $\Jac_X$ is
identified with $\Pic^0_{X/k}$.

\begin{defn}\label{defn:NS}
  The \emph{N\'eron-Severi sheaf\/} of $X$ is the fppf quotient sheaf
  $$\NS_{X/k}=\Pic_{X/k}/\Pic^0_{X/k}.$$
\end{defn}

By construction, the sheaf $\NS_{X/k}$ is representable by an \'etale group
scheme over $k$. (The key point is that the tangent space at the identity
section is trivial, which one can see using the fact that $\NS_{X/k}$ is also
the sheaf of connected components of $\Pic_{X/k}$.) The classical N\'eron-Severi
group $\NS(X)$ is defined as $\Pic(X)/\Pic^0(X)$. With this notation, we see
that $\NS(X)=\NS_{X/k}(k)$ if $k$ is algebraically closed.

\begin{observation}\label{obs:ns}
  Since $\NS_{X/k}$ is \'etale, we have that for any extension
  $L\subset L'$ of separably closed extension fields of $k$, the
  induced map $\NS_{X/k}(L)\to\NS_{X/k}(L')$ is an isomorphism. In
  particular, given a separable closure $k^s$ contained in an
  algebraic closure $\overline{k}$ of $k$, we have $\NS_{X/k}(k^s)
  \cong \NS_{X/k}(\overline{k})$. In particular, any N\'eron-Severi class on $X_{\overline{k}}$ is defined over some finite
  separable extension of $k$. Under the additional assumption
  that $\Pic^0_{X/k}$ is smooth (e.g., for $X$ a curve or an abelian
  variety), we have that any N\'eron-Severi class defined over $k^s$ is induced
  by an invertible sheaf on $X\tensor_k k^s$. (Indeed, in this case the fppf and
  \'etale cohomology of $\Pic^0_{X/k}$ agree by Grothendieck's theorem
  \cite[Th\'eor\`eme 11.7]{MR244271}, so we have that $\H_{\text{\rm fppf}}^1(\Spec k^s,\Pic^0_{X/k})=0$.)
\end{observation}

\subsection{Big monodromy: $\ell$-adic realization}
\label{sec:big-monodromy-ell}

We first show that the N\'eron-Severi group of the Jacobian of the geometric generic fiber
for a curve is generated by the class of the $\Theta$-divisor if the relevant Galois representation
has large image. We then use results of Katz--Sarnak to find families of curves over finite fields
with large monodromy.

\begin{prop}\label{prop:NSgeomgenfiber}
  Let $k$ be any field. Fix a prime $\ell$ invertible in $k$ and assume $k$
  contains all $\ell$-power roots of unity, and fix an isomorphism $\overline \Q_\ell \cong
  \overline \Q_\ell(1)$ of Galois modules.
  Let $C$ over $k$ be a smooth proper geometrically connected
  curve of genus $g$ such that the identity component of the Zariski closure of the image of the Galois representation
  $$\rho_0: \Gal(\overline{k}/k) \to \Sp(2g,\Qlbar)$$
  induced by the Galois action on $\H^1(C_{\overline{k}},\Qlbar)$
  is all of $\Sp(2g,\Qlbar)$.
  Then the geometric generic fiber $\overline{C} := C_{\overline{k}}$ has the
  property that $\NS(\Jac_{\overline{C}})=\Z\Theta$.
\end{prop}

\begin{proof}
    The first Chern class defines a morphism
    $$c: \NS(\Jac_{\overline{C}}) \to \H^2(\Jac_{\overline{C}},\Qlbar),$$
    which is injective modulo torsion.
    (Note that we may ignore Tate twists in this proof since we fixed an isomorphism
    $\overline \Q_\ell \cong \overline \Q_\ell(1)$ above.)
    Because any class of $\NS(\Jac_{\overline{C}})$ is defined
    over a finite separable extension of $k$ by Observation \ref{obs:ns}, the image
    of $c$ is contained in the union of all the subspaces $\H^2(\Jac_{\overline{C}},\Qlbar)^\Gamma$,
    where $\Gamma$ ranges over open subgroups of the
    absolute Galois group $G_k := \Gal(\overline{k}/k)$.
    
    By assumption $G(\rho_0) = \Sp(2g,\Qlbar)$, but also for
    any open subgroup $\Gamma \subset G_k$, the group $G(\rho_0|_\Gamma)$
    is $\Sp(2g,\Qlbar)$, since passage to finite index subgroups does not change
    the identity component of the Zariski closure.
    
    The cup product pairing defines a Galois-invariant map
    $$p: \H^2(\Jac_{\overline{C}},\Qlbar) = \bigwedge\nolimits^2 \H^1(\Jac_{\overline{C}},\Qlbar)  
    = \bigwedge\nolimits^2 \H^1(\overline{C},\Qlbar) \to \Qlbar$$
    with kernel $V$.
    For an open subgroup $\Gamma \subset G_k$, the space
    $\H^2(\Jac_{\overline{C}},\Qlbar)^\Gamma \cap V$ is a subspace
    stable under $\Gamma$, hence under $G(\rho_0|_\Gamma) = \Sp(2g,\Qlbar)$. Since $V$ is an
    irreducible representation of $\Sp(2g,\Qlbar)$ by Section \ref{sec:notation}\eqref{repn-thry},
    the intersection must be $0$.
    
    We thus find that the composition map
    $$\NS(\Jac_{\overline{C}}) \otimes \Qlbar \stackrel{c}{\longrightarrow}
    \bigcup_{\Gamma \subset G_k} \H^2(\Jac_{\overline{C}},\Qlbar)^\Gamma \hookrightarrow
    \H^2(\Jac_{\overline{C}},\Qlbar) \stackrel{p}{\longrightarrow} \Qlbar$$
    is an isomorphism (since $\NS(\Jac_{\overline{C}})$ is not $0$).
    
    As a consequence, we have an injection
    $\NS(\Jac_{\overline{C}}) \hookrightarrow \NS(\overline{C})$.
    It is well known that the pullback of the $\Theta$-divisor class
    to $\overline{C}$ has degree $g$ (see, e.g., \cite[Theorem
    17.4]{MR1987784}), and since $\Theta$ is a principal polarization,
    it is indivisible in $\NS(\Jac_{\overline{C}})$. We therefore have
    $\NS(\Jac_{\overline{C}}) = \Z \Theta$; furthermore, the injection
    $\NS(\Jac_{\overline{C}}) \hookrightarrow \NS(\overline{C})$ is
    identified with $\Z \to g\Z \hookrightarrow \Z$.
\end{proof}

In \cite[Chapter 10]{MR1659828}, Katz and Sarnak produce families of
curves over finite fields with large monodromy groups. Recall that,
given a family of curves $f\colon C\to B$ over a finite field $\Fq$
with $\ell$ an invertible prime, the \emph{geometric monodromy\/}
group $G_{\geom}$ of $f$ is the $\Qlbar$-algebraic group $G(\rho)$
associated to the representation
$$\rho\colon\pi_1(B \otimes_{\Fq} \overline\Fq)\to\Sp(2g,\Qlbar)$$
attached to the lisse sheaf $\R^1f_\ast\Qlbar$ (as in Section
\ref{sec:notation}\eqref{monod-gp}). The following theorem is a
summary of \cite[Theorem 10.1.16 and Theorem 10.2.2]{MR1659828}.

\begin{thm}[Katz--Sarnak]\label{thm:ks}
  For any genus $g$ and any finite field $\Fq$ with $\ell$ an
  invertible prime, there is an open subset $U\subset\A^1_{\Fq}$ and a
  family $\mc C\to U$ of smooth proper geometrically connected genus
  $g$ curves such that $G_\geom=\Sp(2g,\Qlbar)$.
\end{thm}

\begin{cor}\label{cor:big-ell-mono}
  Let $\mc C\to U$ be a family as in Theorem \ref{thm:ks}. The
  geometric generic fiber
  $\overline C:=\mc C\times_U\Spec\overline{\Fq(t)}$ has the property
  that $\NS(\Jac_{\overline C})=\Z\Theta$.
\end{cor}

\begin{proof}
Let $k = \overline{\Fq}(t)$. Theorem \ref{thm:ks} gives that
$G_\geom = G(\rho) = \Sp(2g,\Qlbar)$, and the natural surjection $G_k \cong \pi_1(\Spec k) \to \pi_1(U \otimes_{\Fq} \overline{\Fq})$,
implies that the composition map $\rho_0: G_k \to \pi_1(U \otimes_{\Fq} \overline{\Fq}) \to \Sp(2g,\Qlbar)$ also
has the property that $G(\rho_0) = \Sp(2g,\Qlbar)$.
We thus apply Proposition \ref{prop:NSgeomgenfiber} to obtain the desired result.
\end{proof}

\subsection{Big monodromy: Hodge realization}
\label{sec:big-monodromy-hodge}

We give a briefer sketch of the Hodge version of the argument, as it
seems to be more widespread in the literature. For example, 
\cite[Theorem 17.5.2]{MR2062673} and the discussion leading up to it
are a valuable source.

Let $\mathbf V$ be a local system on a connected space $B$ with
monodromy representation $\rho\colon \pi_1(B,b) \to \GL(V)$, where
$V = \mathbf{V}_b$. Recall that $\mathbf V$ is said to have
\emph{big monodromy\/} if $G(\mathbf V)$ acts irreducibly on $V_{\C}$.

Given a family of principally polarized abelian varieties $g\colon A\to B$,
the polarization defines a quotient sheaf
$\R^2g_\ast\Q=\bigwedge^2\R^1g_\ast\Q\to\underline\Q$. The kernel of
this map is a local system $\mathbf V_2(A)$. Fiberwise, it is
invariant under the symplectic group and is itself an irreducible
representation (Section \ref{sec:notation}\eqref{repn-thry}). We will say that
$g\colon A\to B$ has \emph{big monodromy for $\H^2$} if $\V_2(A)$ has big
monodromy.

\begin{lem}\label{lem:big-mono}
  Let $B$ be a smooth $\C$-scheme. If
  $f\colon X\to B$ is a family of smooth proper 
  curves with $G(\R^1f_\ast\Q)=\Sp(2g,\Q)$, then the Jacobian family
  $g\colon\Jac_{X}=\Pic^0_{X/B}\to B$ has big 
  monodromy for both $\H^1$ and $\H^2$.
\end{lem}
\begin{proof}
  Recall that $\R^1 g_\ast\Q=\R^1f_\ast\Q$ as local systems, so we
  can identify $\R^2 g_\ast\Q$ with $\bigwedge^2\R^1f_\ast\Q$. The rest
  follows from the definitions.
\end{proof}

\begin{cor}\label{cor:big-mono}
  Suppose $B$ is a smooth $\C$-scheme.
  If $f\colon X\to B$ is a smooth proper family of curves with
  $G(\R^1f_\ast\Q)=\Sp(2g,\Q)$, then for any very general $b\in B(\C)$, the
  N\'eron-Severi 
  group $\NS(\Jac_{X_b})$ is isomorphic to $\Z$ and generated by the theta
  divisor $\Theta$.
\end{cor}
\begin{proof}
  As the N\'eron-Severi group for abelian varieties is torsion-free, it is enough
  to show the result after tensoring with $\Q$.
  For any point $b \in B(\C)$, the N\'eron-Severi group
  $\NS(\Jac_{X_b}) \otimes \Q$ is a trivial $\Q$-Hodge structure
  contained in the $\Q$-Hodge structure $\H^2(\Jac_{X_b},\Q)$.  (In
  fact, by the Lefschetz (1,1) theorem, it is the maximal such
  structure.)  By \cite[Theorem 10.20]{peterssteenbrink}, for a very
  general point $b$, the sub-Hodge structure
  $(\V_2)_b \subset \H^2 (\Jac_{X_b},\Q)$ has no rational Hodge
  substructures. By the exact sequence
  $$0  \to (\V_2)_b \to \H^2 (\Jac_{X_b},\Q) \to \Q \to 0,$$
  it follows that $\NS(\Jac_{X_b}) \otimes \Q$ is one-dimensional.
  \end{proof}

\begin{cor}\label{cor:ns-general-curve}
  Let $k$ be a field of characteristic $0$ and $X \to \Spec k$ be a smooth proper
  geometrically connected curve such that the image of the induced
  map $\Spec k \to \ms M_g$ is the generic point. Then we
  have that $\NS(\Jac_{X_{\widebar k}})=\Z\Theta$.
\end{cor}
\begin{proof}
  The statement is invariant under extension of $k$. It is well known
  that for the universal curve over $\ms M_g$ the geometric monodromy
  group attached to
  $\H^1$ is $\Sp(2g)$; see, e.g., \cite[\S 5]{delignemumford} or \cite[Theorem
  6.4]{farbmargalit}. Thus, if $m$ is a very general complex point of
  $\ms M_g$ corresponding to a curve $C$, then by Corollary
  \ref{cor:big-mono} we have that $\NS(\Jac_C)=\Z\Theta$. Since
  N\'eron-Severi groups can only grow under specialization, it follows
  that the geometric generic point must have the same property, giving
  the desired result.
\end{proof}

\subsection{Appeal to example} \label{sec:zarhin}

Since N\'eron-Severi rank can only increase under specialization, another
proof of Proposition \ref{prop:univ-jac} follows from showing that
there exists a single curve of every genus at least $2$, in any characteristic,
whose Jacobian has N\'eron-Severi rank exactly $1$. Results of Zarhin 
\cite{MR1748293, MR2131907} on endomorphism rings of hyperelliptic
Jacobians imply that many such curves exist over most characteristics; 
in fact, Zarhin shows that for any hyperelliptic curve $C: y^2 = f(x)$, where $f(x)$ is an 
irreducible separable degree $n \geq 5$ polynomial with Galois group 
either $S_n$ or $A_n$, in any characteristic $p > 3$,
the endomorphism ring over the algebraic closure is $\Z$, implying
that $\NS(\Jac_C)$ also has rank $1$. These results themselves are quite subtle
and rely on very different techniques than those sketched in this paper.

\subsection{General deformations of smoothable curves}\label{sec:neron-severi-group}

In this section we describe how to put any smoothable curve in a
family with the generic curve. This will be useful for studying the
splitting of Brauer classes, as we explain in Section
\ref{lem:obs}. (We will only apply this to smooth curves, but we
suspect that the full statement for smoothable curves may be useful in
the future, so we record it here.)

\begin{cor}\label{cor:family}
  Suppose $C$ is a smoothable proper geometrically connected curve over a field
  $K$ such that $\Ext^2(L_{C/K},\ms O_C)=0$ (for example, a proper nodal
  curve). Let $W$ be a complete dvr with residue field $K$. There is a
  proper flat family $\mc C\to\Spec W$ such that
  \begin{enumerate}
    \item $\mc C\tensor_W K\cong C$, and
    \item if $\eta\to\Spec W$ is a geometric generic point, then we have
    $\NS(\Jac_{\mc C_\eta})=\Z\Theta$, where $\Theta$ is the usual theta-divisor
    class associated to $\mc C_\eta$.
  \end{enumerate}
\end{cor}
\begin{proof}
  Write $Q$ for the fraction field of $W$. 
  The assumptions on $C$ ensure that the universal formal deformation of $C$ is
  represented by a proper morphism of schemes $\mf C\to\Spec
  W[\![x_1,\ldots,x_n]\!]$ with a smooth fiber. Since the stack of
  proper curves is an Artin stack locally of finite presentation,
  Artin's algebraization theorem
  tells us that there is a pair $(X,x)$ with $X$ a smooth $W$-scheme
  and $x\in X(K)$ and a family $\mc C\to X$ such that the restriction
  of $\mc C\to X$ to $\widehat{\ms O}_{X,x}$ is isomorphic to the universal
  deformation. Let $X^\circ$ denote the locus over which $\mc C$ is
  smooth and let $\F$ be the prime field of $Q$. 
  There is an induced map $\mu\colon X^\circ\to\ms M_{g,\F}$, and by the openness 
  of versality we know that this map is dominant.

  By Proposition \ref{prop:univ-jac}, it suffices to show that
  there is a map $\Spec W\to X$ whose closed point lands at $x$ and
  whose generic point maps to the generic point of $\ms M_{g,\F}$.
  Since $\F$ is countable, $\ms M_{g,\F}$ has only countably many
  closed substacks. Consider the polycylinder
  $B^n=\{(a_0,\ldots,a_n)| a_i\in Q, |a_i|<1\}$ parametrizing all
  $W$-points of $\widehat{\ms O}_{X,x}$. Since $\mu$ is dominant, no
  closed subscheme of $\ms M_{g,\F}$ contains all of $B^n$. Since $Q$ is
  uncountable it follows (e.g., by induction on $n$) that there is a
  point of $B^n$ not in the pullback of any closed substack of $\ms
  M_{g,\F}$. This gives a $W$-point of $X$ with the desired properties.
\end{proof}

\section{Proof of Theorem \ref{thm:main}}

\begin{lem}\label{lem:obs}
  Suppose $C$ is a smooth proper geometrically connected curve of
  genus $g$ over a field $K$ and $C\to 
  X$ is an Albanese morphism. Suppose further that the morphism $\NS(X_{\widebar
  K})\to\NS(C_{\widebar K})$ is injective with image $g\NS(C_{\widebar K})$.
  If $\alpha\in\Br(K)$ is a class of order prime to $g$ such that
  $\alpha_C=0$, then $\alpha_X=0$.
\end{lem}
\begin{proof}
  By the Leray spectral sequence, we have $\alpha_C=0$ if and only if $\alpha$ is the
  obstruction class for a global section of $\Pic_{C/K}$. Consider the diagram
  $$
  \begin{tikzcd}
    0\ar[r] & \Pic^0_{X/K}\ar[r]\ar[d, "\cong"] & \Pic_{X/K}\ar[r]\ar[d] & 
    \NS_{X/K}\ar[r]\ar[d] & 0 \\
    0\ar[r] & \Pic^0_{C/K}\ar[r] & \Pic_{C/K}\ar[r] & \NS_{C/K}\ar[r] & 0
  \end{tikzcd}
$$
of sheaves on $\Spec K$. The Snake Lemma applied to the diagram
yields an exact sequence of sheaves
   $$0\to\Pic_{X/K}\to\Pic_{C/K}\to\Z/g\Z\to 0.$$ 
Thus, the map on global
   sections $$\Pic_{X/K}(K)\to\Pic_{C/K}(K)$$ is injective with cokernel
   annihilated by $g$. Since $g$ is prime to the order of $\alpha$ and the
   obstruction map for sections of the Picard scheme is a group homomorphism, we
   see that $\alpha_X=0$ if and only if $g\alpha_X=0$. If
   $s\in\Pic_{C/K}(K)$ is a section with obstruction $\alpha$, then
   the preimage of $gs$ in $\Pic_{X/K}(K)$ is a section with obstruction $g\alpha$,
   and the desired result follows.
\end{proof}

\subsection*{Proof of Theorem \ref{thm:main}}
  Suppose $C$ is a smooth proper geometrically connected
  curve over $K$ of genus $g\geq 1$ such that $\alpha_C=0$. If
  $g=1$, the desired conclusion holds because $\Alb_C=C$. Thus, we
  assume $g\geq 2$. We will now show that to prove the Theorem it
  suffices to prove that $\Alb_C$ splits $\alpha$ under the additional
  assumption that $g$ is relatively prime to $\ind(\alpha)$ (and
  therefore relatively prime to the order of $\alpha$).
  
  Observe that $\ind(\alpha)$ divides $2g-2 = 2(g-1)$, since the
  canonical divisor of $C$ has degree $2g-2$. As a result, all odd
  divisors of $\ind(\alpha)$ visibly cannot divide $g$, and if
  $4|\ind(\alpha)$, then $g$ is odd. Hence, if
  $\ind(\alpha)\not\equiv 2\pmod 4$, we find that $g$ is relatively
  prime to $\ind(\alpha)$.  If $\ind(\alpha)\equiv 2\pmod 4$, then we can write
  $\alpha=\alpha_2+\alpha'$ with $\alpha_2$ of index $2$ and $\alpha'$
  of odd order. Since $\alpha_2$ is split by a conic (namely, the
  Brauer-Severi variety of the associated division algebra) and any conic
  admits a cover by a genus $1$ curve (namely, the branched cover over
  a general divisor of degree $4$), we see upon taking products
  that it suffices to prove that $\Alb_C$ splits $\alpha'$ (which has
  odd index relatively prime to $g$).

  Let $W$ be a complete dvr with residue field $K$. By \cite[Corollary
  6.2]{MR244269}, since $W$ is a complete dvr, there is a
  unique Brauer class $\widetilde{\alpha}\in\Br(W)$ lifting $\alpha$;
  the lifted class has the same period and index as $\alpha$. By
  Corollary \ref{cor:family}, there is a family $\ms C\to\Spec W$ such
  that the generic fiber $\ms C_\eta$ of $\ms C$ satisfies the
  conditions of Lemma \ref{lem:obs}.  Since $\alpha_C=0$, the
  deformation theory of invertible twisted sheaves as in
  \cite{MR2388554} (or \cite[Theorem 3.2]{tateymctateface}, which is
  recorded without proof) tells us that
  $\widetilde\alpha_{\ms C_\eta}=0$. By Lemma \ref{lem:obs}, we have
  that $\widetilde\alpha_{\Alb_{\ms C_\eta}}=0$.  Since
  $\Alb_{\ms C/W}$ is regular, it follows (for example,
  \cite[Corollary 1.10]{MR244270}) that
  $\widetilde\alpha_{\Alb_{\ms C/W}}=0$. Thus, $\alpha_{\Alb_C}=0$ by
  specialization.\footnote{Instead of using a very general $W$-point
    of the universal deformation ring of $C$ over $K$, we could just
    use the universal deformation directly. Our approach avoids
    enlarging the fraction field of $W$ at the expense of a small
    amount of extra work.}

\subsection*{Antieau and Auel's proof of Theorem \ref{thm:main}}
A different proof of Theorem \ref{thm:main}, due to Benjamin Antieau
and Asher Auel \cite{auelantieau-private}, uses results on the stable birational
geometry of symmetric powers of Brauer--Severi varieties to 
show that an appropriate symmetric power of a curve splitting a 
Brauer class also splits the class. Here is a sketch of their proof; more
details may appear elsewhere in the future. 

Suppose $C$ is a smooth proper geometrically connected curve over $K$ of genus
$g \geq 1$ such that $\alpha_C = 0$.  As in the first proof, we may reduce to
showing that $\Alb_C$ splits $\alpha$ under the assumptions that $\alpha$ is
non-trivial and that $g$ is relatively prime to $\ind(\alpha)$.

The image of $C$ in a Brauer-Severi variety $V$ associated to $\alpha$ cannot be
a point, so the image of the induced map $\Sym^{2g-1} C \to \Sym^{2g-1} V$
intersects the smooth locus of $\Sym^{2g-1} V$. By \cite[Theorem 1
(4)]{kollar-symSeveriBrauer} (see also \cite{MR2090670}), the space
$\Sym^{2g-1} V$ is stably birational to $\Sym^m V$, where
$m = \gcd(2g-1, \ind(\alpha))$; here, we have $m=1$ since $\ind(\alpha)$ divides
$2g-2$. Thus, the smooth locus $U$ of $\Sym^{2g-1} V$ is stably birational to
$V$ and splits $\alpha$, which implies that $\Sym^{2g-1} C$ also splits
$\alpha$.

By the Riemann-Roch theorem, we have that
$\sigma \colon \Sym^{2g-1} C \to \Pic^{2g-1}_{C/K} \cong \Alb_C$ is a
Brauer--Severi scheme of relative dimension $g-1$. Since
$\sigma^\ast\alpha_{\Alb_C}=0$, the class $\alpha_{\Alb_C}$ is $g$-torsion. But
$\alpha_{\Alb_C}$ is also killed by $\ind(\alpha)$, so the assumption that $g$
is relatively prime to $\ind(\alpha)$ implies that $\alpha_{\Alb_C}=0$, as
desired.

\section{The conditions of Theorem \ref{thm:main} are
  necessary}\label{sec:counterex}

In this section, we show that there are many examples of Brauer classes
$\alpha$ with index congruent to $2$ modulo $4$
that split on a curve $C$ but not on $\Alb_C$. These examples are
easily constructed over local fields (and hence over many finitely
generated fields, by standard approximation techniques), and we
suspect one could also make similar examples over number fields.

Given a smooth proper geometrically connected curve $C$ over a field
$K$, recall that the \emph{index\/} of $C$ is the smallest degree of a
divisor on $C$, and the \emph{period\/} of $C$ is the smallest degree
of a divisor class. (Equivalently, the period of $C$ is the order of
the Albanese variety $\Alb_C=\Pic^1_{C/K}$ in
$\H^1(\Spec K, \Jac_C)$.)

\begin{lem}\label{lem:stupid-die}
  Suppose $C$ is a curve over a field $K$ and $\alpha\in\Br(K)$ is a class of
  order $2$. If $\alpha_{\Pic^1_C}=0$, then $\alpha_{\Pic^m_C}=0$ for any odd
  number $m$.
\end{lem}
\begin{proof}
  The $m$th power map on the Picard stack descends to a morphism
  $$\lambda_m\colon \Pic^1_C\to\Pic^m_C$$
  that is an \'etale form of the multiplication by $m$ on $\Jac_C$. In
  particular, $\lambda_m$ is finite flat of degree $m^{2g}$, where $g$
  is the genus of $C$. By assumption, the class $\alpha_{\Pic^m_C}$
  vanishes upon pullback along the morphism $\lambda_m$ of odd
  degree. By standard calculations in Galois cohomology (see, for
  example, \cite[Proposition 4.1.1.1]{MR2388554}), this implies that
  $\alpha_{\Pic^m_C}=0$, as desired.
\end{proof}

\begin{prop}\label{prop:counterex}
  Let $m$ be an odd positive integer and suppose $C$ is a smooth proper
  geometrically connected curve over a local
  field $K$ of index $2m$, period $m$, and genus $m+1$. Then the unique
  non-zero Brauer class $q\in\Br(K)[2]$ is killed by $C$ but not by $\Alb_C$.
  Thus, there are Brauer classes $\alpha$ of all even indices  
  dividing $2m$ that are killed by $C$ but not by $\Alb_C$.
\end{prop}
\begin{proof}
  Since $m$ is odd, any class $\alpha$ in $\Br(K)[2m]$ can be written as $q+h$
  with $h\in\Br(K)[m]$. The relative Brauer group $\Br(C/K)$ is
  precisely $\Br(K)[2m]$ by the theorem of Roquette--Lichtenbaum
  \cite[Theorem 3]{lichtenbaum-duality}. In addition, period and index
  are equal over a local field. Thus, to
  prove the full statement, it suffices to prove the first part.
  
  From Roquette--Lichtenbaum, we have that $q_C=0$. If $q_{\Alb_C}=0$ also,
  then $q_{\Pic^m_C}=0$ by Lemma \ref{lem:stupid-die}. Since $C$ has
  period $m$, there is a $K$-point of $\Pic^m_C$, and restricting to that point
  would imply that $q = 0$, which is a contradiction. 
\end{proof}

By a result of Sharif \cite[Theorem 2]{sharif}, for a local field $K$ of characteristic not
$2$ and for any odd $m$, there exists a curve $C$ over $K$ of index $2m$, period $m$,
and genus $m+1$. Proposition \ref{prop:counterex} then shows that the Albanese
of $C$ cannot kill numerous classes in $\Br(C/K)$, implying that the conditions
of Theorem \ref{thm:main} are sharp.

\section{Some observations}

\subsection{Products of genus one curves}

One might attempt to answer Question \ref{ques:main}
by first splitting the class on a family of Albanese varieties with a
member that splits as a product of genus $1$ curves, and hoping that
this decomposition will have implications for splitting the class over
a factor. As we briefly explain, there are two reasons that this is
unlikely to work.

First, the results of Section \ref{sec:counterex} show that one cannot
hope to use only Albanese varieties of curves from the beginning, because there are
examples where the decisive role is played by a genus $1$ factor added
after the fact,
whose presence is necessary to split a single quaternion factor of the
Brauer class.

Second, we make a simple observation about products: suppose $T$
and $T'$ are genus $1$ curves over 
$K$, with $T$ of index $2$ and $T'$ of index $3$. Any Brauer class $\alpha$
killed by $T$ has order $2$ and any Brauer class $\alpha'$ killed by $T'$ has
order $3$. The natural tensoring map $\Pic_T\times\Pic_{T'}\to\Pic_{T\times T'}$
is additive on obstruction classes (since it is equivariant for the
multiplication map $\G_m\times\G_m\to\G_m$ of bands for the Picard stacks).
Thus, $\alpha+\alpha'$ is killed by $T\times T'$, which is a torsor under
$\Jac_T\times\Jac_{T'}$, but $\alpha+\alpha'$ is not killed by either $T$ or
$T'$.

Examples of both types are easily constructed over local fields.

\subsection{The universal Albanese doesn't do anything on its own}
\label{sec:an-amusing-fact}

In light of the method used here -- that is, splitting Brauer classes
by splitting them on particular base changes of the Albanese map of the
universal curve -- one might be tempted to ask the following question.

\begin{question}\label{ques:idiot}
  Given a field $K$ and a positive integer $g>2$, let $C\to\Spec\kappa(\ms
  M_{g,K})$ be the universal curve of genus $g$ over the function
  field of the stack of all curves of genus $g$. What is the kernel of
  the map $$\Br(K)\to\Br(C)?$$
\end{question}

\begin{prop}
  The kernel of the map of Question \ref{ques:idiot} is $0$. 
\end{prop}
\begin{proof}
  Over any field $K$, there is a curve $C_0$ of genus $g$ with a
  $K$-point. Since the universal curve is regular, any class that is
  trivialized over $C$ is trivialized over any specialization, such as
  $C_0$. Further specializing to the $K$-point shows that the
  class itself is $0$. 
\end{proof}

It might be more interesting to study the relative Brauer group of the
universal curve over its field of definition.

\section{Some questions}
\label{sec:questions}

Some natural questions arise from the results we describe here.
As mentioned in the introduction, one way to produce curves splitting
Brauer classes is as complete intersections of sections
of the anticanonical sheaf in a Brauer-Severi variety. 
On the other hand, if we restrict to a single anticanonical divisor,
we obtain Calabi-Yau varieties that split the class. 
This observation leads to several directions for further exploration.

\begin{question}
  Is there a fixed positive integer $n$ such that every Brauer class
  over a field is split by a torsor under an abelian variety of
  dimension $n$, independent of the index of the class? (Note that
  Theorem \ref{thm:main} applied to complete intersections of
  anti-canonical divisors gives a torsor of dimension $1+\frac{1}{2}m^{m-1}(m-3)$ for classes
  of index $m$, but this depends on $m$.)
\end{question}

\begin{question}
  Is there a fixed positive integer $n$ such that every Brauer class
  over a field is split by a Calabi-Yau variety of dimension $n$?
\end{question}

\begin{question}
  Is every Brauer class over a field split by a K3 surface?
\end{question}

\begin{question}
  Is every Brauer class over a field split by a curve sitting in a
  K3 surface?
\end{question}

\bibliographystyle{amsalpha}
\bibliography{bib}

\providecommand{\bysame}{\leavevmode\hbox to3em{\hrulefill}\thinspace}
\providecommand{\MR}{\relax\ifhmode\unskip\space\fi MR }
\providecommand{\MRhref}[2]{%
  \href{http://www.ams.org/mathscinet-getitem?mr=#1}{#2}
}
\providecommand{\href}[2]{#2}
\begin{thebibliography}{BKLV18}

\bibitem[AA18]{auelantieau-private}
Benjamin Antieau and Asher Auel, private communication, 2018.

\bibitem[Aue15]{auel-video}
Asher Auel, \emph{Algebras of composite degree split by genus one curves},
  \url{http://www.birs.ca/events/2015/5-day-workshops/15w5016/videos/watch/201509151021-Auel.html},
  2015.

\bibitem[BKLV18]{1711.07722}
Francesco Bastianelli, Alexis Kouvidakis, Angelo~Felice Lopez, and Filippo
  Viviani, \emph{Effective cycles on the symmetric product of a curve, {I}: the
  diagonal cone (with an appendix by {Ben Moonen})},
  \url{https://arxiv.org/abs/1711.07722}, 2018.

\bibitem[BL04]{MR2062673}
Christina Birkenhake and Herbert Lange, \emph{Complex abelian varieties},
  second ed., Grundlehren der Mathematischen Wissenschaften [Fundamental
  Principles of Mathematical Sciences], vol. 302, Springer-Verlag, Berlin,
  2004. \MR{2062673}

\bibitem[dJH12]{ajdjwh}
Aise~Johan de~Jong and Wei Ho, \emph{Genus one curves and {B}rauer-{S}everi
  varieties}, Math. Res. Lett. \textbf{19} (2012), no.~6, 1357--1359.
  \MR{3091612}

\bibitem[DM69]{delignemumford}
P.~Deligne and D.~Mumford, \emph{The irreducibility of the space of curves of
  given genus}, Inst. Hautes \'Etudes Sci. Publ. Math. (1969), no.~36, 75--109.
  \MR{0262240}

\bibitem[FH91]{fultonharris}
William Fulton and Joe Harris, \emph{Representation theory}, Graduate Texts in
  Mathematics, vol. 129, Springer-Verlag, New York, 1991, A first course,
  Readings in Mathematics. \MR{1153249}

\bibitem[FM12]{farbmargalit}
Benson Farb and Dan Margalit, \emph{A primer on mapping class groups},
  Princeton Mathematical Series, vol.~49, Princeton University Press,
  Princeton, NJ, 2012. \MR{2850125}

\bibitem[Gro68a]{MR244269}
Alexander Grothendieck, \emph{Le groupe de {B}rauer. {I}. {A}lg\`ebres
  d'{A}zumaya et interpr\'etations diverses}, Dix expos\'es sur la cohomologie
  des sch\'emas, Adv. Stud. Pure Math., vol.~3, North-Holland, Amsterdam, 1968,
  pp.~46--66. \MR{244269}

\bibitem[Gro68b]{MR244270}
\bysame, \emph{Le groupe de {B}rauer. {II}. {T}h\'eorie cohomologique}, Dix
  expos\'es sur la cohomologie des sch\'emas, Adv. Stud. Pure Math., vol.~3,
  North-Holland, Amsterdam, 1968, pp.~67--87. \MR{244270}

\bibitem[Gro68c]{MR244271}
\bysame, \emph{Le groupe de {B}rauer. {III}. {E}xemples et compl\'ements}, Dix
  expos\'es sur la cohomologie des sch\'emas, Adv. Stud. Pure Math., vol.~3,
  North-Holland, Amsterdam, 1968, pp.~88--188. \MR{244271}

\bibitem[Kol18]{kollar-symSeveriBrauer}
J\'anos Koll\'ar, \emph{Symmetric powers of {S}everi-{B}rauer varieties},
  \url{https://arxiv.org/abs/1603.02104}, 2018.

\bibitem[KS99]{MR1659828}
Nicholas~M. Katz and Peter Sarnak, \emph{Random matrices, {F}robenius
  eigenvalues, and monodromy}, American Mathematical Society Colloquium
  Publications, vol.~45, American Mathematical Society, Providence, RI, 1999.
  \MR{1659828}

\bibitem[KS04]{MR2090670}
Daniel Krashen and David~J. Saltman, \emph{Severi-{B}rauer varieties and
  symmetric powers}, Algebraic transformation groups and algebraic varieties,
  Encyclopaedia Math. Sci., vol. 132, Springer, Berlin, 2004, pp.~59--70.
  \MR{2090670}

\bibitem[Lic69]{lichtenbaum-duality}
Stephen Lichtenbaum, \emph{Duality theorems for curves over {$p$}-adic fields},
  Invent. Math. \textbf{7} (1969), 120--136. \MR{0242831}

\bibitem[Lie08]{MR2388554}
Max Lieblich, \emph{Twisted sheaves and the period-index problem}, Compos.
  Math. \textbf{144} (2008), no.~1, 1--31. \MR{2388554}

\bibitem[Moo18]{moonen-appendix}
Ben Moonen, \emph{Appendix to ``{E}ffective cycles on the symmetric product of
  a curve, {I}: the diagonal cone''}, \url{https://arxiv.org/abs/1711.07722},
  2018.

\bibitem[Pol03]{MR1987784}
Alexander Polishchuk, \emph{Abelian varieties, theta functions and the
  {F}ourier transform}, Cambridge Tracts in Mathematics, vol. 153, Cambridge
  University Press, Cambridge, 2003. \MR{1987784}

\bibitem[PS08]{peterssteenbrink}
Chris A.~M. Peters and Joseph H.~M. Steenbrink, \emph{Mixed {H}odge
  structures}, Ergebnisse der Mathematik und ihrer Grenzgebiete. 3. Folge. A
  Series of Modern Surveys in Mathematics [Results in Mathematics and Related
  Areas. 3rd Series. A Series of Modern Surveys in Mathematics], vol.~52,
  Springer-Verlag, Berlin, 2008. \MR{2393625}

\bibitem[RV11]{openproblems}
Anthony Ruozzi and Uzi Vishne, \emph{Open problem session from the conference
  {``Ramification in algebra and geometry''}},
  \url{http://www.mathcs.emory.edu/RAGE/RAGE-open-problems.pdf}, 2011.

\bibitem[Sha07]{sharif}
Shahed Sharif, \emph{Curves with prescribed period and index over local
  fields}, J. Algebra \textbf{314} (2007), no.~1, 157--167. \MR{2331756}

\bibitem[Swe95]{Swets}
Paul~Kenneth Swets, \emph{Global sections of higher powers of the twisting
  sheaf on a {B}rauer-{S}everi variety}, ProQuest LLC, Ann Arbor, MI, 1995,
  Thesis (Ph.D.)--The University of Texas at Austin. \MR{2693834}

\bibitem[Tat68]{tateymctateface}
John Tate, \emph{On the conjectures of {B}irch and {S}winnerton-{D}yer and a
  geometric analog [see {MR}1610977]}, Dix expos\'es sur la cohomologie des
  sch\'emas, Adv. Stud. Pure Math., vol.~3, North-Holland, Amsterdam, 1968,
  pp.~189--214. \MR{3202555}

\bibitem[Zar00]{MR1748293}
Yuri~G. Zarhin, \emph{Hyperelliptic {J}acobians without complex
  multiplication}, Math. Res. Lett. \textbf{7} (2000), no.~1, 123--132.
  \MR{1748293}

\bibitem[Zar04]{MR2131907}
\bysame, \emph{Non-supersingular hyperelliptic {J}acobians}, Bull. Soc. Math.
  France \textbf{132} (2004), no.~4, 617--634. \MR{2131907}

\end{thebibliography}

\end{document}